\documentclass{article}
\usepackage[T1]{fontenc}
\usepackage[english]{babel}
\usepackage[margin=1.in]{geometry}
\usepackage[utf8]{inputenc}
\usepackage{lmodern}
\usepackage{microtype}
\usepackage{dsfont}

\usepackage{amsmath,amssymb,amsfonts,amsthm}
\usepackage{mathtools}
\usepackage{mathrsfs} 
\usepackage[square,comma,numbers]{natbib}
\usepackage[colorlinks=true, allcolors=blue]{hyperref}
\usepackage{graphicx}
\usepackage{enumerate}
\usepackage[font=small]{caption}
\usepackage{subcaption}
\usepackage[dvipsnames]{xcolor}
\usepackage{tikz}

\numberwithin{equation}{section}

\def\cE{{\mathcal E}}
\def\cF{{\mathcal F}}

\def\cL{{\mathcal L}}

\def\cP{{\mathcal P}}

\def\cQ{{\mathcal Q}}

\def\E{\mathbb{E}}
\def\F{\mathbb{F}}

\def\P{\mathbb{P}}
\def\R{\mathbb{R}}
\def\U{\mathbb{U}}

\def\T{\mathbb{T}}
\def\Z{\mathbb{Z}}

\def\sA{\mathscr{A}}

\def\sF{\mathscr{F}}

\def\Lip{\mathrm{Lip}}

\renewcommand{\d}{\mathrm{d}}
\newcommand{\dr}{\mathrm{d}r}
\newcommand{\ds}{\mathrm{d}s}
\newcommand{\dt}{\mathrm{d}t}

\newcommand{\dx}{\mathrm{d}x}
\newcommand{\dy}{\mathrm{d}y}

\newcommand{\eps}{\epsilon}
\newcommand{\Ent}{\mathrm{Ent}}
\newcommand{\BL}{\mathrm{BL}}

\newtheorem*{Def*}{Definition}
\newtheorem*{Thm*}{Theorem}
\newtheorem*{Cor*}{Corollary}
\newtheorem*{Rmk*}{Remark}
\newtheorem*{Lem*}{Lemma}
\newtheorem*{Prop*}{Proposition}
\newtheorem*{Asm*}{Assumption}

\newtheorem{Def}{Definition}[section]
\newtheorem{Thm}[Def]{Theorem}
\newtheorem{Cor}[Def]{Corollary}

\newtheorem{Lem}[Def]{Lemma}
\newtheorem{Prop}[Def]{Proposition}
\newtheorem{Asm}[Def]{Assumption}

\def\be{\begin{equation}}
\def\ee{\end{equation}}

\usepackage{fancyhdr}
\title{Convergence of Potential Mean-Field Games \\via Lyapunov Methods\footnote{Research partially supported 
by the National Science Foundation grant DMS 2406762. The author is grateful to H.~Mete Soner and Kevin Zhang for helpful comments and discussions. ChatGPT provided useful editorial assistance.}}
\author
{Felix H\"ofer\footnote{Department of Operations Research and Financial
Engineering, Princeton University, Princeton, NJ, 08540, USA}}

\date{}

\begin{document}
\maketitle

\begin{abstract}
\noindent 
We consider discounted infinite-horizon potential mean-field games (MFGs) on the $d$-dimensional torus. Without imposing monotonicity assumptions, we prove that every weak limit point of a time-dependent equilibrium, as time tends to infinity, is a stationary equilibrium. As a consequence, equilibria converge whenever the stationary solution is unique. The short proof is based on a novel Lyapunov functional for the time-dependent MFG system. We also provide a new uniqueness criterion for stationary equilibria. Finally, we apply our results to the subcritical Kuramoto MFG studied by Carmona, Cormier, and Soner, showing that every equilibrium converges to the incoherent solution.
\end{abstract}

\section{Introduction}

The theory of mean-field games (MFGs), initiated by Lasry and Lions \cite{LL1} and Huang, Caines, and Malham\'{e} \cite{HMC1}, provides powerful techniques for studying symmetric, non-cooperative large-population games. We refer to the monographs by Carmona and Delarue \cite{carmona2018probabilistic} and Ba{\c{s}}ar, Djehiche, and Tembine \cite{bacsar2026mean} for comprehensive introductions. In this work, we investigate the long-time behavior of time-dependent equilibria and the uniqueness of stationary equilibria for discounted infinite-horizon MFGs. Questions of this type are central both in macroeconomic theory \cite{achdou2014partial, achdou2022income} and in the broader analysis of general, possibly non-monotone, mean-field games. Despite considerable recent interest in non-monotone MFGs, these questions remain largely open for the problem considered here. We provide some answers in the \emph{potential case}, that is, for MFGs that admit an associated central-planner mean-field control (MFC) problem whose minimizers are MFG equilibria, see H\"{o}fer and Soner \cite{hofer2024optimal}.

Let us briefly introduce the problem. Consider a filtered probability space $(\Omega,\sF,\F=(\sF_t)_{t\ge0},\P)$ carrying a $d$-dimensional $\F$-Brownian motion $(W_t)_{t\ge0}$, and let $\sA$ denote the set of all $\F$-progressively measurable $d$-dimensional processes. Let $\rho>0$ be a discount factor and $\nu>0$. Given a measure flow $\boldsymbol{\mu}=(\mu_t)_{t\ge0}$ on the torus $\T^d=(\R/2\pi\Z)^d$, a typical player solves the stochastic control problem (SCP):
\begin{equation}\label{eq:typical-player-problem}\tag{SCP$(\boldsymbol{\mu})$}
\left\{
\begin{split}
\inf_{\alpha\in\sA}\qquad &\E\Big[\int_0^\infty e^{-\rho t}\Big(\frac12 |\alpha_t|^2 + f(X^\alpha_t,\mu_t)\Big)\,\dt \Big]\\
\text{s.t.}\qquad &\d X^\alpha_t = \alpha_t\,\dt + \sqrt{2\nu}\,\d W_t,\qquad X^\alpha_0\sim\mu_0,
\end{split}
\right.
\end{equation}
where we assume that the stochastic basis is rich enough to support an $\sF_0$-measurable random variable with law $\mu_0$. The state equation is regarded modulo $2\pi$, so that the state space is $\T^d$. A \emph{solution} or \emph{equilibrium of the MFG starting from a distribution $\mu\in\cP(\T^d)$} is a measure flow $\boldsymbol{\mu}=(\mu_t)_{t\ge0}$ with $\mu_0=\mu$ such that there exists an optimal control $\alpha^*\in\sA$ in \eqref{eq:typical-player-problem} satisfying $\cL(X^{\alpha^*}_t)=\mu_t$ for all $t\ge0$. A \emph{stationary solution}, or \emph{stationary equilibrium of the MFG}, is a measure $\mu\in\cP(\T^d)$ such that the constant flow $\mu_t\equiv\mu$ is a solution of the MFG starting in $\mu$. We ask the questions:
\begin{equation}\label{main-question-i}
\text{How do time-dependent equilibria $\boldsymbol{\mu}=(\mu_t)_{t\ge0}$ behave as $t\uparrow\infty$?}
\vspace{0.5em}
\end{equation}
and 
\begin{equation}\label{main-question-ii}
\text{When is a stationary equilibrium unique?}
\vspace{0.5em}
\end{equation}
Whenever there is a unique stationary solution, it is natural to expect time-dependent equilibria to converge to it, and we confirm this hypothesis in our setting. Our approach relies on identifying a suitable Lyapunov functional of the time-inhomogeneous MFG system, and we hope this technique will prove useful in other contexts. We outline the main ideas and state our main results in Theorems \ref{thm:main} and \ref{thm:main-stationary} below. 

Under Assumption \ref{asm:coupling} on the coupling cost $f$ below, a measure flow $\boldsymbol{\mu}$ is a solution to the MFG if and only if there exists a corresponding solution to the associated \emph{MFG system}, consisting of a backward Hamilton-Jacobi-Bellman (HJB) equation coupled with a forward Fokker-Planck-Kolmogorov equation on $(0,\infty)\times\T^d$:
\begin{equation}\label{eq:inhomog-MFG-system}\tag{MFG$_\rho$}
\left\{
\begin{aligned}
&-\partial_t u +\rho u - \nu \Delta u + \frac12 |\nabla u|^2 = f(x,\mu_t),\\
& \partial_t \mu -\nu \Delta \mu - \nabla\cdot (\mu\nabla u)=0,\qquad \mu_0\ \text{given},
\end{aligned}
\right.
\end{equation}
where the unknowns are functions $u:[0,\infty)\times\T^d\mapsto\R$ and $\mu:[0,\infty)\mapsto\cP(\T^d)$. The HJB equation is understood in the viscosity sense and the Fokker-Planck-Kolmogorov equation in the weak sense. However, for positive times, all solutions in this paper are classical under our Assumption \ref{asm:coupling} below, and we identify $\mu_t(\dx) =m(t,x)\,\dx$ with its density for $t>0$. We write $\cP^2_+(\T^d)$  for the set of all twice continuously differentiable, strictly positive densities. Stationary solutions $\mu$ are in one-to-one correspondence with solutions to the \emph{stationary MFG system} on $\T^d$:
\begin{equation}\label{eq:stationary-MFG-system}\tag{MFG$_{\rho}^\infty$}
\left\{
\begin{aligned}
&\rho u - \nu \Delta u + \frac12 |\nabla u|^2 = f(x,\mu), \\
&-\nu \Delta  m - \nabla\cdot (m\nabla u)=0,
\end{aligned}
\right.
\end{equation}
where $\mu(\dx)=m(x)\,\dx$, and the unknowns are $(u,m):\T^d\mapsto\R\times[0,\infty)$. With a slight abuse of terminology, we will interchangeably refer to a measure flow $(\mu_t)_{t\ge0}$, its flow of densities $(m_t)_{t>0}$, or the corresponding solution $(u_t,\mu_t)_{t\ge0}$ or $(u_t,m_t)_{t>0}$ as an \emph{equilibrium} or \emph{solution of the MFG}, with a similar convention for the stationary case.

Our main assumption is that the interaction cost $f$ is a linear derivative $f=\delta_\mu \cF$ of some potential $\cF:\cP(\T^d)\mapsto\R$. In this case, we introduce a (generalized) \emph{free energy}
\begin{equation}\label{eq:free-energy}
\Phi(\mu) :=\rho\, \nu \, \Ent(\mu)+  \frac{\nu^2}{2}I(\mu) + \cF(\mu),\qquad \mu\in\cP^2_+(\T^d).
\end{equation}
Here, for $\mu(\dx)=m(x)\,\dx$, 
$$
\Ent(\mu):=\int_{\T^d} \log (m(x)) \, m(x)\,\dx,\qquad I(\mu):= \int_{\T^d} \frac{|\nabla m(x)|^2}{m(x)}\,\dx
$$
denote the \emph{entropy} and the \emph{Fisher information}, respectively. When $\rho=0$, the map $\Phi$ reduces to a Fisher-information-regularized, or \emph{Ginzburg–Landau}, energy, see \cite{claisse2026mean}. A direct computation (see Lemma \ref{lem:Phi-derivative}) shows that $\mu(\dx)=m(x)\,\dx$ is a stationary equilibrium if and only if it is a stationary point of $\Phi$, in the sense that its linear derivative $\delta_\mu\Phi(\mu,\cdot)$ is constant. Next, given a time-dependent equilibrium $(u,m)$, the main idea is to introduce a new variable
$$
q(t,x)=u(t,x)+\nu\log(m(t,x)),\qquad (t,x)\in(0,\infty)\times\T^d,
$$
and consider the functional 
\begin{equation*}
\cL(t):=\Phi(\mu_t) -  \frac12 \cQ(t),\qquad \cQ(t):=\int_{\T^d} |\nabla q(t,x)|^2m(t,x)\,\d x.
\end{equation*}
It turns out that $\cL$ acts as a Lyapunov functional for the system \eqref{eq:inhomog-MFG-system} and satisfies $\dot\cL=-\rho\cQ\leq0$, see Proposition \ref{prop:Lyapunov}.

\begin{Thm}[Main result]\label{thm:main}
Let Assumption \ref{asm:coupling} below be in force, and let $\boldsymbol{\mu}=(\mu_t)_{t\ge0}$ be any time-dependent equilibrium. Then:
\begin{enumerate}[(i)]
\item\label{eq:main-result-i} For any $T>0$, $\lim_{t\uparrow\infty}\sup_{s\in[0,T]}\d_\BL(\mu_t,\mu_{t+s})=0$.
\item\label{eq:main-result-ii} Any weak limit point of $\boldsymbol{\mu}$ as $t\uparrow\infty$ is a stationary equilibrium. In particular, if there exists a unique stationary equilibrium $\mu^\star\in\cP(\T^d)$, then $\boldsymbol{\mu}$ converges to $\mu^\star$ as $t\uparrow\infty$.
\end{enumerate}
Here, convergence is understood with respect to weak convergence of measures, and $\d_\BL$ denotes the bounded Lipschitz metric on $\cP(\T^d)$.
\end{Thm}

We emphasize that we do not assume any monotonicity conditions for Theorem \ref{thm:main}. We also remark that this result does not imply the convergence of any time-dependent equilibrium by itself. However, Theorem \ref{thm:main} \eqref{eq:main-result-i} rules out periodic solutions with strictly positive period. The metric $\d_\BL$ metricizes weak convergence, and on $\cP(\T^d)$ it is equivalent to the Wasserstein-1 (or Kantorovich–Rubinstein) metric.

To formulate our result on the uniqueness of stationary solutions, let us call a family of maps $S_t:\cP^2_+(\T^d)\mapsto\cP^2_+(\T^d)$, $t\ge0$, an \emph{admissible test flow} if $S_0\mu=\mu$ for all $\mu\in\cP^2_+(\T^d)$, $\|S_h\mu-\mu\|_{C^2}\to0$ as $h\downarrow0$, and if there exists a map $V:\cP^2_+(\T^d)\mapsto C(\T^d)$ such that $\|h^{-1}(S_h\mu-\mu)-V\mu\|_\infty\to 0$ as $h\downarrow0$. An example of an admissible test flow is the heat flow $S_t=e^{t\Delta}$, and we consider this special case below.

\begin{Thm}\label{thm:main-stationary}
Let $(S_t)_{t\ge0}$ be an admissible test flow. Then, any stationary equilibrium $\mu\in\cP^2_+(\T^d)$ satisfies 
\begin{equation}\label{eq:sufficient-not-equilibrium}
\frac{\d}{\dt}\Big|_{t=0} \Phi(S_t \mu)=0.
\end{equation}
\end{Thm}

Taking $(S_t)_{t\ge0}$ to be the heat flow $(e^{t\Delta})_{t\ge0}$, we provide the following necessary condition for measures to be stationary equilibria.

\begin{Cor}[Heat flow]\label{cor:heat-flow}
Any stationary equilibrium $\mu\in\cP^2_+(\T^d)$ satisfies
$$
\frac{\d}{\dt}\Big|_{t=0}\cF(\eta^\mu_t) \geq \nu(\rho+\nu)I(\mu),\qquad \text{where} \ \ (\eta^\mu_t)_{t\ge0} \ \ \text{solves} \ \ \partial_t \eta^\mu_t=\Delta\eta^\mu_t,\ \eta^\mu_0=\mu.
$$
\end{Cor}

We are particularly motivated by the works of Carmona, Cormier, and Soner \cite{carmona_synchronization_2023, carmona2025kuramoto} who study a model of collective synchronization, the so-called \emph{Kuramoto MFG}. Inspired by the classical Kuramoto model, agents aim to synchronize their phases (positions) by minimizing a misalignment cost, and the model undergoes a phase transition as this cost increases. In the subcritical regime, Carmona, Cormier, and Soner \cite{carmona_synchronization_2023} are able to prove the local stability of the uniform, or \emph{incoherent}, equilibrium in which phases are distributed uniformly. As an application of our results, we establish both the uniqueness of the uniform equilibrium and the convergence of any time-dependent equilibrium to it in the subcritical regime, see Section \ref{sec:Kuramoto}.

The term ``long-time behavior'' of a mean-field game usually refers to the convergence of finite-horizon games to a stationary MFG problem, see Cardaliaguet, Lasry, Lions, and Porretta \cite{long-time-local,long-time-nonlocal} for both local and non-local MFGs, and Cardaliaguet and Porretta \cite{long-time-master} for an approach based on the Nash-Lasry-Lions equation (also known as the master equation). Gomes, Mohr, and Souza \cite{gomes2010discrete} treat this problem in the discrete state case. \cite{long-time-master} also considers the limit as the discount factor goes to zero. Let us also note that the references above rely on Lasry-Lions monotonicity. Cirant and Porretta  \cite{cirant2021long} extended the analysis to  mildly non-monotone MFGs. Without monotonicity, Cardaliaguet and Masoero \cite{cardaliaguet2020weak} study this problem for potential MFGs and develop a weak KAM theory. In the non-monotone case, Cirant \cite{cirant2019existence} and Cirant and Nurbekyan \cite{cirant2018variational} show the existence of time-periodic solutions. We refer to Carmona, Tangpi, and Zhang \cite{kevin} and Bayraktar and Zhang \cite{bayraktar2023solvability} for probabilistic analyses. Finally, we remark that the free energy $\Phi$ also plays an important role in mean-field optimization, see Claisse et al.~\cite{claisse2026mean}. However, in our problem, the correction $\cQ$ is crucial.

\vspace{1em}
Throughout this manuscript, we impose the following assumption on $f$.

\begin{Asm}\label{asm:coupling}
The interaction cost $f$ satisfies the following.
\begin{enumerate}[(i)]
\item \emph{Potential structure.} There exists $\cF:\cP(\T^d)\mapsto\R$ such that $f=\delta_\mu\cF$. The functional $\cF$ is bounded from below. Here, $\delta_\mu$ denotes the linear derivative.
\item \emph{Regularity.} $f$ is twice continuously differentiable in $x$, and once continuously differentiable in $\mu$ with derivative $\delta_\mu f(x,\mu,y)$. We assume that 
$$
\sup_{\mu\in\cP(\T^d)}\|f(\cdot,\mu)\|_{C^2(\T^d)}<\infty,
$$
and that $\nabla_x\delta_\mu f$, $\nabla_{y}\delta_\mu f$, $\partial_{x_i}\partial_{y_j}\delta_\mu f$, and $\Delta_y \delta_\mu f$ all exist and are continuous on $\T^d\times\cP(\T^d)\times\T^d$.
\end{enumerate}
\end{Asm}

\noindent

\vspace{0.3em}

For future use, we remark that, under Assumption \ref{asm:coupling}, it is classical that, for any $\mu_0\in\cP(\T^d)$, the system \eqref{eq:inhomog-MFG-system} admits a classical solution on $(0,\infty)\times\T^d$. Due to non-degenerate Brownian noise, the density $m$ is strictly positive on $(0,\infty)\times\T^d$. We also note that the assumptions on $f$ are sufficient to justify that mixed derivatives in $t$ and $x$ commute. We refer to \cite{long-time-master} and \cite[Lemma 4.5]{carmona_synchronization_2023} for details. In Appendix \ref{app:HJB-estimates}, we derive $L^\infty$ bounds for the solutions. Finally, let us note that by elliptic regularity, it is also classical that \eqref{eq:stationary-MFG-system} admits a classical solution, see Bardi and Feleqi \cite{bardi2016nonlinear}. In this case, the stationary density takes the form, for $x\in\T^d$,
\begin{equation}\label{eq:stationary-density}
m(x)= \frac{\exp(-u(x)/\nu)}{Z},\qquad Z=\int_{\T^d} \exp(-u(y)/\nu)\,\dy.
\end{equation}

\vspace{1em}

{\bf Structure.} We prove the statements on the long-time behavior recorded in Theorem \ref{thm:main} \eqref{eq:main-result-i} and \eqref{eq:main-result-ii} in Section \ref{sec:proof-main}. In Section \ref{sec:uniqueness-stationary}, we discuss the uniqueness of stationary equilibria, and, in particular, prove Theorem \ref{thm:main-stationary} and Corollary \ref{cor:heat-flow}. Section \ref{sec:Kuramoto} applies our results to the Kuramoto MFG, and the Appendices \ref{app:HJB-estimates}--\ref{app:fisher-info-lower-bound} collect auxiliary results.

\vspace{1em}
{\bf Notation.} For a vector $x\in\R^d$, $|x|:=(x_1^2+\ldots+x_d^2)^{1/2}$ denotes its Euclidean norm. The set of non-negative integers is denoted by $\Z_+=\{0,1,2,\ldots\}$. The set of probability measures on $\T^d$ is denoted by $\cP(\T^d)$ and is endowed with the topology of weak convergence. The \emph{bounded Lipschitz distance} on $\cP(\T^d)$ is defined by
$$
\d_\BL(\mu,\nu):=\sup\big\{|\mu(f)-\nu(f)|:f\in C(\T^d)\  \text{with} \ \|f\|_{\infty}+\Lip(f)\leq1\big\},\qquad \mu,\nu\in\cP(\T^d).
$$
Here, the integral of a function $f:\T^d\mapsto\R$ with respect to a measure $\mu\in\cP(\T^d)$ is denoted by $\mu(f):=\int f\,\d\mu$, and $\Lip(f)$ denotes the optimal Lipschitz constant of the function $f$.
It is classical that $\d_\BL$ metrizes weak convergence and that $(\cP(\T^d),\d_\BL)$ is a compact metric space. As in the introduction, we set
$$
\cP^2_+(\T^d) := \big\{m:\T^d\mapsto(0,\infty) \, \big|\, m \ \text{is twice continuously differentiable and} \ \int_{\T^d} m\,\dx =1\big\},
$$
and, with a slight abuse of notation, we also write $\mu\in\cP^2_+(\T^d)$ whenever $\mu\in\cP(\T^d)$ has a Lebesgue density $m\in\cP^2_+(\T^d)$.
We call a function $\phi:\cP(\T^d)\mapsto\R$ \emph{linearly differentiable} if there exists a map $\delta_\mu\phi:\cP(\T^d)\times\T^d\mapsto\R$ such that, for every $\mu,\nu\in\cP(\T^d)$,
$$
\phi(\nu)-\phi(\mu)=\int_0^1\int_{\T^d}\delta_\mu\phi((1-\tau)\mu+\tau\nu,x)(\nu-\mu)(\dx)\d\tau.
$$
The linear derivative $\delta_\mu\phi(\mu,x)$ is only defined up to a $\mu$-dependent constant. Finally, derivatives at time $0$ are understood as right-sided derivatives.

\section{Long-time behavior}\label{sec:proof-main}

This section establishes the results on the long-time behavior of equilibria stated in Theorem \ref{thm:main}, parts \eqref{eq:main-result-i} and \eqref{eq:main-result-ii}. As explained in the introduction, our proof relies on defining a Lyapunov functional $\cL$ of the MFG system. Now fix a time-dependent solution $(u,m)=(u(t,\cdot),m(t,\cdot))_{t>0}$ of the MFG system \eqref{eq:inhomog-MFG-system} with an arbitrary initial condition $\mu_0\in\cP(\T^d)$. We set $\mu_t(\dx):=m(t,x)\,\dx$. Recall $\Phi$ from \eqref{eq:free-energy}.

\begin{Lem}[Linear derivative of $\Phi$]\label{lem:Phi-derivative}  
The linear derivative of $\Phi$ is given by
$$
\delta_\mu\Phi(\mu,x) =  \nu \rho \log m(x)-\nu^2 \Delta \log m(x)-\frac{\nu^2}{2} |\nabla\log m(x)|^2 + f(x,\mu),\qquad (\mu,x)\in \cP^2_+(\T^d)\times\T^d,
$$
where $m$ denotes the density of $\mu$. In particular, $\mu$ is a stationary equilibrium if and only if $\delta_\mu\Phi(\mu,\cdot)$ is a constant independent of $x$.
\end{Lem}

\begin{proof}
An explicit calculation of the linear derivative can be found in \cite[Proposition 2.4]{claisse2026mean}. The claim about the stationary equilibria follows from the formula \eqref{eq:stationary-density}.
\end{proof}
 Next, set
\begin{equation}\label{eq:w-and-q}
w = - \nu \log(m),\qquad q := u - w,\qquad \text{on}\ (0,\infty)\times\T^d,
\end{equation}
and, as outlined in the introduction,
\begin{equation}\label{eq:Lyapunov-def}
\cE(t):=\Phi(\mu_t),\qquad \cQ(t):=\int_{\T^d}|\nabla q(t,x)|^2\, m(t,x)\,\d x,\qquad \cL(t):=\cE(t)-\frac12 \cQ(t),\qquad t>0.
\end{equation}

The main observation is that $\cL$ acts as a Lyapunov functional for the system \eqref{eq:inhomog-MFG-system}, see Proposition \ref{prop:Lyapunov}. To prove it, we first need a short lemma.

\begin{Lem}
On $(0,\infty)\times\T^d$,
\begin{equation}\label{eq:m-change}
\partial_t m = \nabla \cdot(m\nabla q)
\end{equation}
and
\begin{equation}\label{eq:w-change}
\partial_t w = \nabla q\cdot \nabla w - \nu\Delta q.
\end{equation}
\end{Lem}

\begin{proof}
First note that $\nabla w=-\nu \nabla m/m$. Therefore,
$$
\partial_t m = \nu \Delta m+ \nabla\cdot(m\nabla u) = \nu \Delta m+ \nabla \cdot(m (\nabla q+\nabla w)) = \nu \Delta m+ \nabla \cdot(m (\nabla q-\nu \nabla m/m)) = \nabla \cdot(m\nabla q).
$$
Next, 
\begin{equation*}
\partial_t w = -\nu\frac{\partial_t m}{m} = -\nu \frac{\nabla\cdot(m\nabla q)}{m}= -\nu \frac{\nabla m\cdot \nabla q}{m} - \nu \Delta q = \nabla w\cdot \nabla q-\nu \Delta q. \qedhere
\end{equation*}
\end{proof}

\vspace{1em}

\begin{Prop}\label{prop:Lyapunov} 
The functional $\cL:(0,\infty)\mapsto\R$ is bounded from below and satisfies
\begin{equation}\label{eq:Lyapunov}
\dot \cL(t) = -\rho \cQ(t)\leq0.
\end{equation}

\end{Prop}

\begin{proof}
By definition of $w$, $\log m=-w/\nu$. By the chain rule and Lemma \ref{lem:Phi-derivative},
\begin{align*}
\dot \cE(t) &=
\frac{\d}{\dt} \Phi(\mu_t) \\
&=  \int_{\T^d} \left(\nu \rho \log m-\nu^2 \Delta \log m-\frac{\nu^2}{2} |\nabla\log m|^2 + f(x,\mu_t)  \right)\partial_tm(t,x)\,\d x\\
&=  \int_{\T^d} \left(-\rho w+\nu \Delta w-\frac12 |\nabla w|^2 + f(x,\mu_t)  \right)\partial_tm(t,x)\,\d x.
\end{align*}
Now,
\begin{align*}
0&= \rho u -\partial_t u-\nu\Delta u+\frac12 |\nabla u|^2 - f(x,\mu_t)\qquad \text{(since $(u,m)$ satisfies \eqref{eq:inhomog-MFG-system})}\\
&= \rho (q+w) - \partial_t (q+w) - \nu \Delta(q+w) +\frac12 |\nabla q+\nabla w|^2 -f(x,\mu_t) \qquad \text{(since $u=q+w$)}\\
&=\rho q - \partial_t q  +\frac12 |\nabla q|^2 + \rho w - \nu \Delta w + \frac12 |\nabla w|^2 -f(x,\mu_t) + \underbrace{\big(-\partial_t w -\nu\Delta q +\nabla q\cdot \nabla w\big)}_{=0\ \text{by \eqref{eq:w-change}}}.
\end{align*}
Hence,
$$
\dot\cE(t)= \int_{\T^d}\Big(\rho q -\partial_t q + \frac12 |\nabla q|^2 \Big)\,\partial_t m\,\dx.
$$
Using \eqref{eq:m-change} and integration by parts,
\begin{align*}
\rho \int_{\T^d} q \partial_t m\,\dx = \rho \int_{\T^d} q \nabla \cdot(m\nabla q)\,\dx = -\rho \int_{\T^d} |\nabla q|^2m\,\dx = -\rho \cQ(t).
\end{align*}
Similarly,
$$
-\int_{\T^d} (\partial_t q )( \partial_t m) \,\dx = -\int_{\T^d} (\partial_t q )\nabla \cdot(m\nabla q) \,\dx = \int_{\T^d} \nabla (\partial_t q)\cdot \nabla q\, m \,\dx
$$
Now
$$
\dot \cQ(t) = \frac{\d}{\dt}\int_{\T^d} |\nabla q|^2m\,\dx=\int_{\T^d} |\nabla q|^2 \partial_t m \,\dx + 2 \int_{\T^d} \nabla (\partial_t q)\cdot \nabla q \, m\,\dx,
$$
so that
$$
-\int (\partial_t q )( \partial_t m) \,\dx =  \frac12 \dot \cQ(t) -  \frac12 \int |\nabla q|^2 \partial_t m\,\dx.
$$
In total, this establishes $\dot \cE(t) = -\rho \cQ(t)+\frac12 \dot \cQ(t)$ since the last term cancels. Hence
$$
\dot \cL(t) = \dot\cE(t) - \frac12 \dot \cQ(t) = -\rho \cQ(t),
$$
as claimed in \eqref{eq:Lyapunov}.
Next, 
\begin{align*}
\cQ(t)&= \int_{\T^d} |\nabla u - \nabla w|^2m\,\d x\\
&= \int_{\T^d} \Big|\nabla u + \nu \frac{\nabla m}{m}\Big|^2 m\,\dx \qquad \text{(since $\nabla w = -\nu \nabla m/m$)}\\
&= \int_{\T^d} |\nabla u|^2 m\,\dx + 2\nu \int_{\T^d}  \nabla u\cdot \nabla m \,\dx + \nu^2 \int_{\T^d} \frac{|\nabla m|^2}{m}\,\dx\\
&\leq \|\nabla u(t,\cdot)\|_{L^\infty(\T^d)}^2 + 2 \nu \|\Delta u(t,\cdot)\|_{L^\infty(\T^d)} + \nu^2 I(\mu_t).
\end{align*}
Using $x\log x\geq x-1$ for $x\ge0$,
$$
\Ent(\mu)=\int_{\T^d} \log (m) m\,\dx \geq \int_{\T^d} (m-1)\,\dx \geq - (2\pi)^d,
$$
and recall that $\cF$ is lower bounded by Assumption \ref{asm:coupling}. This leads to
$$
\cE(t) = \frac{\nu^2}{2}I(\mu_t) + \rho \nu \, \Ent(\mu_t) +\cF(\mu_t)\geq \frac{\nu^2}{2}I(\mu_t) -\rho\nu(2\pi)^d + \inf \cF.
$$
At the same time, by Lemma \ref{lem:L-infty-bounds}, there exists a constant $c_0>0$ such that 
$$
\sup_{t\geq0} \frac12 \|\nabla u(t,\cdot)\|^2_\infty + \nu \|\Delta u(t,\cdot)\|_\infty\leq c_0.
$$
Combining this with the bound for $\cQ(t)$, we obtain a lower bound for $\cL$:
\begin{equation*}
\frac12 \cQ(t)\leq  \cE(t) - \inf \cF + c_0 + \rho\nu(2\pi)^d\quad \Rightarrow \quad \cL(t)=\cE(t)-\frac12 \cQ(t) \geq \inf\cF -c_0-\rho\nu(2\pi)^d>-\infty.\qedhere
\end{equation*}
\end{proof}

\begin{Cor}\label{cor:Q-integrable}
The function $\cQ:[1,\infty)\mapsto[0,\infty)$ is Lebesgue-integrable:
\begin{equation}\label{eq:Q-integrable}
    \int_1^\infty \cQ(t)\,\d t<\infty.
\end{equation}
\end{Cor}
\begin{proof}
Integrating \eqref{eq:Lyapunov} yields, for any $T\ge 1$,
\begin{align*}
\rho \int_1^T \cQ(s)\,\ds = \cL(1)-\cL(T) \leq \cL(1)-\inf \cL(\cdot).
\end{align*}
Since $\rho>0$, taking $T\uparrow \infty$ shows that $\cQ\in L^1([1,\infty))$.
\end{proof}

The next corollary establishes Theorem \ref{thm:main} \eqref{eq:main-result-i}.

\begin{Cor}\label{cor:mu-shift}
For any $T>0$,
\begin{equation}\label{eq:mu-shift}
\lim_{t\uparrow\infty}\sup_{s\in[0,T]}\,\d_\BL(\mu_t,\mu_{t+s})=0.
\end{equation}
\end{Cor}
\begin{proof}
It suffices to establish
\begin{equation}\label{eq:mu-shift-proof}
\d_\BL(\mu_s,\mu_t)\leq |t-s|^{1/2} \left(\int_s^t \cQ(u)\,\d u\right)^{1/2}
\end{equation}
for any $1\le s\leq t$. Indeed, we then conclude, 
$$
\sup_{s\in[0,T]}\,\d_\BL(\mu_t,\mu_{t+s}) \leq T^{1/2}\left(\int_1^\infty \cQ(u)\, \mathds{1}_{[t,t+T]}(u)\,\d u\right)^{1/2}\longrightarrow\, 0,\qquad \text{as} \ t\uparrow\infty,
$$
by dominated convergence since $\cQ\in L^1([1,\infty))$ by Corollary \ref{cor:Q-integrable}.

Let $\varphi:\T^d\mapsto\R$ be a bounded Lipschitz function, which we may, without loss of generality, take to be smooth. We compute, for $t>0$,
\begin{align*}
\frac{\d}{\dt} \mu_t(\varphi) &= \int_{\T^d} \varphi\, \partial_tm\,\dx\\
&=\int_{\T^d} \varphi \, \nabla\cdot(m\nabla q)\,\dx\qquad \text{(from \eqref{eq:m-change})}\\
&=-\int_{\T^d} \,\nabla \varphi\cdot \nabla q\,m\,\dx\qquad \text{(integration by parts)}.
\end{align*}
Hence, for $1\leq s\leq t$,
\begin{align*}
|\mu_{t}(\varphi)-\mu_s(\varphi)|&\leq \int_s^{t} \Big|\frac{\d}{\d u} \mu_u(\varphi) \Big|\,\d u \\
&\leq \int_s^{t} \int_{\T^d} | \nabla \varphi(x)\cdot \nabla q(u,x)|\, m(u,x)\,\dx\,\d u\\
&\leq  \Lip(\varphi)\ \int_s^t \left(\int_{\T^d} |\nabla q(u,x)|^2 \, m(u,x)\,\d x\right)^{1/2}\,\d u\qquad \text{(Cauchy-Schwarz)}\\
&= \Lip(\varphi)\, \int_s^t \sqrt{\cQ(u)}\,\d u\\
&\leq \Lip(\varphi) \, |t-s|^{1/2} \left(\int_s^t \cQ(u)\,\d u\right)^{1/2}\qquad \text{(Cauchy-Schwarz)}.
\end{align*}
Taking the supremum over $\varphi$ shows \eqref{eq:mu-shift-proof}, as claimed.
\end{proof}

The next Proposition \ref{prop:limit-points-are-solutions} establishes the first claim recorded in Theorem \ref{thm:main} \eqref{eq:main-result-ii}. To prove it, we note the next lemma which follows by joint continuity of $f:\T^d\times\cP(\T^d)\mapsto\R$ and compactness of $\cP(\T^d)$.

\begin{Lem}\label{lem:f-continuous}
$f$ is uniformly continuous viewed as a map $(\cP(\T^d),\d_\BL)\mapsto (C(\T^d),\|\cdot\|_\infty)$.
\end{Lem}


\begin{Prop}\label{prop:limit-points-are-solutions}
Any weak limit point of $(\mu_t)_{t\ge0}$ is a stationary equilibrium.
\end{Prop}

\begin{proof}
Let $t_n\uparrow\infty$ be any sequence such that $\d_\BL(\mu_{t_n}, \bar \mu)\to0$ for some $\bar \mu\in\cP(\T^d)$.
For $(s,x)\in(0,\infty)\times\T^d$ and $n\ge1$ define 
$$
\psi_n(s,x):=f(x,\mu_{t_n+s}),\quad \bar \psi(x):=f(x,\bar\mu),\quad u_n(s,x):=u(t_n+s,x), \quad m_n(s,x):=m(t_n+s,x).
$$
Since $(u,m)$ satisfy \eqref{eq:inhomog-MFG-system}, $u_n$ satisfies the following dynamic programming equation
$$
-\partial_s u_n+\rho u_n-\nu \Delta u_n + \frac12 |\nabla u_n|^2 = \psi_n(s,x),\qquad (s,x)\in (0,\infty)\times\T^d.
$$
By Lemma \ref{lem:f-continuous}, there exists a modulus of continuity $\omega_f(\cdot)$ such that $\|f(\cdot,\mu)-f(\cdot,\nu)\|_{L^\infty(\T^d)}\leq\omega_f(\d_\BL(\mu,\nu))$. Hence, 
$$
\sup_{s\in[0,T]} \|\psi_n(s,\cdot)-\bar \psi\|_\infty=\sup_{s\in[0,T]}  \|f(\cdot,\mu_{t_n+s})- f(\cdot,\bar\mu)\|_{\infty}\leq \omega_f\Big(\sup_{s\in[0,T]}\d_\BL(\mu_{t_n+s},\bar \mu)\Big)\longrightarrow0,
$$
as $n\uparrow\infty$. Here, the last step follows from the triangle inequality
$$
\sup_{s\in[0,T]}\d_\BL(\mu_{t_n+s},\bar \mu)\leq \sup_{s\in[0,T]}\d_\BL(\mu_{t_n+s},\mu_{t_n})+\d_\BL(\mu_{t_n},\bar\mu),
$$
which goes to zero by Corollary \ref{cor:mu-shift} and the assumption. Hence $\psi_n\to\bar \psi$ locally uniformly as $n\uparrow\infty$. Furthermore,
$
\nabla \psi_n(s,x) = \nabla f(x,\mu_{t_n+s}),
$
which is uniformly bounded in $(s,x)\in[0,\infty)\times\T^d$ and $n\ge1$. Also,
\begin{align*}
\partial_s \psi_n(s,x) &= \int_{\T^d} \delta_\mu f(x,\mu_{t_n+s},y)\partial_s m_n(s,y)\,\d y\\
&=\int_{\T^d} \delta_\mu f(x,\mu_{t_n+s},y)(\nu\Delta m_n(s,y) + \nabla\cdot(m_n\nabla u_n))\,\d y\\
&=  \int_{\T^d} (\nu \Delta_y \delta_\mu f(x,\mu_{t_n+s},y)  -  \nabla_y \delta_\mu f(x,\mu_{t_n+s},y) \cdot \nabla_y u_n)\, m_n(s,y)\d y,
\end{align*}
which is uniformly bounded by Assumption \ref{asm:coupling} and Lemma \ref{lem:L-infty-bounds}. By Lemma \ref{lem:stability}, $\nabla u_n\to\nabla \bar u$ locally uniformly on $[0,\infty)\times\T^d$ where $\bar u$ solves
$$
\rho \bar u - \nu \Delta \bar u + \frac12 |\nabla \bar u|^2 = \bar\psi(x)=f(x,\bar \mu).
$$
It remains to show that $\bar \mu$ is a stationary equilibrium. Let $\varphi\in C^\infty(\T^d)$ and fix $T>0$. Then,
\begin{align*}
\mu_{t_n+T}(\varphi)-\mu_{t_n}(\varphi) &= \int_0^T \frac{\d}{\ds} \mu_{t_n+s}(\varphi)\,\ds\\
&=\int_0^T \int_{\T^d} \varphi(x)\,\partial_s m_n(s,x)\,\dx\,\ds\\
&= \int_0^T \int_{\T^d} \varphi(x)\,(\nu \Delta m_n(s,x) + \nabla\cdot(m_n(s,x)\nabla u_n(s,x)))\,\dx\,\ds\\
&=\int_0^T \int_{\T^d}(\nu\Delta \varphi(x)-\nabla\varphi(x)\cdot\nabla u_n(s,x))\,m_n(s,x)\,\dx\,\ds.
\end{align*}
The left side goes to $0$ by Corollary \ref{cor:mu-shift}. We claim that the right side converges to $T\bar\mu(\nu\Delta\varphi-\nabla\varphi\cdot\nabla\bar u)$. Let 
$$
\eta_n(s,x):=\nu\Delta \varphi(x)-\nabla\varphi(x)\cdot\nabla u_n(s,x),\qquad \bar \eta(x):=\nu\Delta\varphi(x)-\nabla\varphi(x)\cdot\nabla\bar u(x),\qquad (s,x)\in[0,\infty)\times\T^d.
$$
By Lemma \ref{lem:stability}, we have 
$$
\lim_{n\uparrow\infty} \, \|\eta_n-\bar\eta\|_{L^\infty([0,T]\times\T^d)}=0,\qquad \forall \, T>0.
$$
Since $\bar\eta$ is bounded Lipschitz,
$$
\sup_{s\in[0,T]}|\mu_{t_n+s}(\bar\eta)-\bar\mu(\bar\eta)|\leq (\|\bar \eta\|_{L^\infty(\T^d)}+\Lip(\bar\eta))\ \sup_{s\in[0,T]} \d_\BL(\mu_{t_n+s},\bar\mu)\longrightarrow0,
$$
as $n\uparrow\infty$. Therefore,
\begin{align*}
&\Big|\int_0^T \int_{\T^d}\eta_n(s,x)\,m_n(s,x)\,\dx\,\ds - T\bar \mu(\bar \eta)\Big| \leq T\|\eta_n-\bar\eta\|_{L^\infty([0,T]\times\T^d)} + T\sup_{s\in[0,T]}|\mu_{t_n+s}(\bar\eta)-\bar\mu(\bar\eta)| \longrightarrow0,
\end{align*}
as $n\uparrow \infty$.
This shows that, for any $\varphi\in C^\infty(\T^d)$,
$$
\bar\mu(\nu\Delta\varphi-\nabla\varphi\cdot \nabla \bar u) =\bar\mu(\bar\eta)=0,
$$
implying that $(\bar u,\bar\mu)$ is a weak solution to \eqref{eq:stationary-MFG-system}. Since weak solutions are classical, the claim follows.
\end{proof}

\vspace{1em}

We provide the proof of the remaining convergence claim.
\begin{proof}[Proof of Theorem \ref{thm:main} \eqref{eq:main-result-ii}]
Assume there exists a unique stationary equilibrium $\mu^\star$. Since $\cP(\T^d)$ is compact, any subsequence has a further convergent subsequence.
By Proposition \ref{prop:limit-points-are-solutions}, any limiting point is a stationary equilibrium. Since this must be equal to $\mu^\star$, we have $\d_\BL(\mu_t,\mu^\star)\to0$ as $t\uparrow\infty$.
\end{proof}

\section{Uniqueness of stationary equilibria}\label{sec:uniqueness-stationary}

When are stationary equilibria, or equivalently, solutions to \eqref{eq:stationary-MFG-system} unique? We first present general criteria: the well-known Lasry-Lions monotonicity condition and an abstract criterion on the potential $\cF$ recorded in Theorem \ref{thm:main-stationary}. We then specialize to heat flows in Subsection \ref{ssec:heat-flow} and apply our criteria to a class of energy functionals in Subsection \ref{ssec:energy-functionals}. 

\subsection{Lasry-Lions monotonicity}\label{ssec:LL-monotone}

\begin{Def}
{\rm
The interaction $f:\T^d\times\cP(\T^d)\mapsto\R$ is \emph{Lasry-Lions monotone} if, for every $\mu,\nu\in \cP(\T^d)$,
$$
\int_{\T^d} (f(x,\mu)-f(x,\nu))\,(\mu-\nu)(\dx)\geq0.
$$
}
\end{Def}

\begin{Prop}
If $f$ is Lasry-Lions monotone, then there is at most one stationary equilibrium.
\end{Prop}

The proof appears to be standard although we were not able to find a precise reference, so we include it for completeness.

\begin{proof}
Let $(u_i,m_i)$, $i=1,2$, be two stationary solutions and set $\mu_i(\dx)=m_i(x)\,\dx$. We claim that
\begin{equation}\label{eq:ll-duality}
\begin{aligned}
\int_{\T^d} (f(x,\mu_1)-f(x,\mu_2))\,(\mu_1-\mu_2)(\dx) =&-\rho\nu\int_{\T^d} (\log(m_1)-\log(m_2)) (m_1-m_2)\,\d x\\
& -\frac12 \int_{\T^d}(m_1+m_2)|\nabla u_1-\nabla u_2|^2\,\dx
\end{aligned}
\end{equation}
Since $\log(\cdot)$ is strictly increasing, this is a strictly negative quantity if $\mu_1\neq \mu_2$. Then, under Lasry-Lions monotonicity, we must have $\mu_1=\mu_2$. Set $w=u_1-u_2$ and $z=m_1-m_2$. Subtracting the stationary Fokker-Planck equations and multiplying by $w$ gives, after integration by parts,
\begin{equation}\label{eq:ll-proof-I}
\nu \int_{\T^d}\nabla z\cdot \nabla w\,\dx + \int_{\T^d} (m_1\nabla u_1 - m_2 \nabla u_2)\cdot \nabla w\,\dx=0.
\end{equation}
At the same time, subtracting the stationary HJB equations and multiplying by $z$,
\begin{equation}\label{eq:ll-proof-II}
\int_{\T^d} (f(\cdot,\mu_1)-f(\cdot,\mu_2))z\,\dx = \rho \int_{\T^d} wz\,\dx + \nu \int_{\T^d} \nabla z\cdot \nabla w\,\dx +\frac12 \int_{\T^d} z(|\nabla u_1|^2 - |\nabla u_2|^2)\,\dx.
\end{equation}
Combining \eqref{eq:ll-proof-I} and \eqref{eq:ll-proof-II},
\begin{align*}
\int_{\T^d} (f(\cdot,\mu_1)-f(\cdot,\mu_2))z\,\dx &=\rho \int_{\T^d} wz\,\dx  - \int_{\T^d} (m_1\nabla u_1 - m_2 \nabla u_2)\cdot \nabla w\,\dx+ \frac12 \int_{\T^d} (|\nabla u_1|^2 - |\nabla u_2|^2)z \,\dx.
\end{align*}
Using 
$$
 -  (m_1\nabla u_1 - m_2 \nabla u_2)\cdot \nabla w +\frac12 (|\nabla u_1|^2 - |\nabla u_2|^2)z = -\frac12 (m_1+m_2)|\nabla w|^2
$$
we obtain
\begin{align*}
\int_{\T^d} (f(\cdot,\mu_1)-f(\cdot,\mu_2))z\,\dx &=\rho \int_{\T^d} wz\,\dx -\frac12 \int_{\T^d} (m_1+m_2)|\nabla w|^2\,\dx.
\end{align*}
By \eqref{eq:stationary-density},
$$
m_i(x)=\frac{1}{Z_i} \exp(-u_i(x)/\nu),\qquad Z_i:=\int_{\T^d} \exp(-u_i(y)/\nu)\,\d y,\qquad i=1,2,
$$
so that, by taking logarithms,
$$
\int_{\T^d}wz\,\dx = -\nu\int_{\T^d} (\log(m_1)-\log(m_2)) (m_1-m_2)\,\d x.
$$
This completes the proof.
\end{proof}

\subsection{Abstract flow criterion}\label{ssec:abstract-flow}

This section establishes Theorem \ref{thm:main-stationary}. From Lemma \ref{lem:Phi-derivative}, we know that $\mu\in\cP^2_+(\T^d)$ is a stationary equilibrium if and only if $\delta_\mu\Phi(\mu,\cdot)\equiv\text{constant}$. Let $V$ be the derivative of $S_t$ at time $0$, see Theorem \ref{thm:main-stationary}. Assume $\mu$ is a stationary equilibrium. By our regularity assumptions on $(S_t)_{t\ge0}$ and $V$ and the chain rule,
$$
\frac{\d}{\dt}\Big|_{t=0} \Phi(S_t\mu)=\int_{\T^d} \delta_\mu \Phi(\mu,x)\, (V\mu)(x)\,\d x =\text{constant } \cdot \int_{\T^d} (V\mu)(x)\,\d x.
$$
But clearly, 
$$
\int_{\T^d} (V\mu)(x)\,\dx = \lim_{h\downarrow0} \frac{1}{h} \int_{\T^d} (S_t\mu-\mu)(x)\,\dx=0.
$$
This proves Theorem \ref{thm:main-stationary}.

\subsection{Heat flow criterion}\label{ssec:heat-flow}

We now specialize to the heat flow $S_t=e^{t\Delta}$ and establish Corollary \ref{cor:heat-flow}. Define the constant
\begin{equation}\label{eq:critical-kappa}
\kappa_c:=2\nu(\rho+\nu).    
\end{equation}

\begin{Lem}\label{lem:torus-inequality}
For any smooth strictly positive function $g:\T\mapsto\R$, 
$$
\int_\T |(\log g)'|^2 g\,\dx \leq \int_\T |(\log g)''|^2 g\,\dx.
$$
\end{Lem}
The simple proof is presented in Appendix \ref{app:torus-inquality}. The next lemma in particular proves Corollary \ref{cor:heat-flow}.

\begin{Lem}\label{lem:heat-flow-identities}
Fix $\mu\in\cP^2_+(\T^d)$, and let $(\eta^\mu_t)_{t\ge0}$ satisfy the heat equation $\partial_t\eta^\mu_t=\Delta\eta^\mu_t$ with $\eta^\mu_0=\mu$. Then,
\begin{equation}\label{eq:heat-flow-identities-I}
\frac{\d}{\dt}\Ent(\eta^\mu_t)=-I(\eta^\mu_t)
\end{equation}
and 
\begin{equation}\label{eq:heat-flow-identities-II}
\frac{\d}{\dt} I(\eta^\mu_t)\le -2 I(\eta^\mu_t).
\end{equation}
In particular, if $\mu$ is a stationary equilibrium, then
\begin{equation}\label{eq:heat-flow-identities-III}
\frac{\d}{\dt}\Big|_{t=0} \cF(\eta^\mu_t)\geq \frac{\kappa_c}{2} I(\mu).
\end{equation}
Here, $\kappa_c$ is as in \eqref{eq:critical-kappa}.
\end{Lem}

\begin{proof}
\emph{Proof of \eqref{eq:heat-flow-identities-I}.} This is a consequence of integration by parts, and it is also known as \emph{de Bruijn's identity}. For a proof, see \cite[Proposition 5.2.2]{bakry} for example.

\vspace{1em}
\emph{Proof of \eqref{eq:heat-flow-identities-II}.} Let us write $\eta=\eta^\mu$. It is classical (see p.~269 in \cite{bakry}) that
$$
\frac{\d}{\dt} I(\eta_t) = -2 \int_{\T^d}|D^2 \log(\eta_t)|_{\mathrm{HS}}^2\,\eta_t\,\d x,
$$
where $|\cdot|_{\mathrm{HS}}$ is the Hilbert-Schmidt norm and $D^2$ denotes the Hessian. Applying  Lemma \ref{lem:torus-inequality} to each coordinate, and integrating over the remaining coordinates, yields
$$
\int_{\T^d} |\nabla \log(\eta_t)|^2\,\eta_t\,\dx\leq \int_{\T^d} \sum_{i=1}^d \Big| \frac{\partial^2}{\partial x_i^2}\, \log(\eta_t)\Big|^2 \,\eta_t\, \dx.
$$
Hence,
\begin{align*}
\frac{\d}{\dt} I(\eta_t) &= -2 \int_{\T^d}|D^2 \log(\eta_t)|_{\mathrm{HS}}^2\, \eta_t\,\d x\\
&\leq -2\int_{\T^d} \sum_{i=1}^d \Big| \frac{\partial^2}{\partial x_i^2}\, \log(\eta_t)\Big|^2 \,\eta_t\, \dx \\
&\leq -2\int_{\T^d} |\nabla \log(\eta_t)|^2\,\eta_t\,\dx=-2I(\eta_t).
\end{align*}

\emph{Proof of \eqref{eq:heat-flow-identities-III}.} 
Assume $\mu$ is a stationary equilibrium. By Theorem \ref{thm:main-stationary},
\begin{align*}
0&= \frac{\d}{\dt}\Big|_{t=0}  \Phi(\eta_t)\\
&=\rho \nu \frac{\d}{\dt}\Big|_{t=0} \Ent(\eta_t) + \frac{\nu^2}{2} \frac{\d}{\dt}\Big|_{t=0} I(\eta_t) + \frac{\d}{\dt}\Big|_{t=0} \cF(\eta_t) \\
&\leq -\rho \nu I(\mu) - \nu^2 I(\mu) +\frac{\d}{\dt}\Big|_{t=0} \cF(\eta_t),
\end{align*}
hence
$$
\frac{\d}{\dt}\Big|_{t=0} \cF(\eta^\mu_t) \geq \nu(\rho+\nu) I(\mu) = \frac{\kappa_c}{2} I(\mu)
$$
by definition \eqref{eq:critical-kappa} of $\kappa_c$, showing the claim.
\end{proof}

\subsection{Uniqueness for a class of energy functionals}\label{ssec:energy-functionals}
We now specialize to potentials $\cF$ of the form
\begin{equation}\label{eq:F-kernel}
\cF(\mu):= \iint_{\T^d\times\T^d}\psi(x-y)\,\mu(\dx)\mu(\d y),\qquad \mu\in\cP(\T^d),
\end{equation}
for a smooth even function $\psi:\T^d\mapsto\R$, so that the interaction cost $f(x,\mu)$ of the MFG is given by, up to $\mu$-dependent constants,
\begin{equation}\label{eq:f-kernel}
f(x,\mu):=2 \ \int_{\T^d}\psi(x-y)\mu(\d y),\qquad (x,\mu)\in\T^d\times\cP(\T^d).
\end{equation}
Observe that the uniform distribution $\U(\dx)=(2\pi)^{-d}\,\dx$ on the torus is a stationary solution to the MFG. Indeed, by translation-invariance of $\U$, 
$$
f(x,\U) = \frac{2}{(2\pi)^d}\int_{\T^d} \psi(y)\,\d y
$$
is independent of $x$, so that the optimal control is $\alpha^\star\equiv0$. Started in the uniform distribution, the optimal state remains uniform since it is a Brownian motion on $\T^d$. We now provide criteria under which $\U$ is the \emph{unique} stationary equilibrium.

Since $\psi$ is smooth and even, it is represented by its Fourier series
$$
\psi(x) = c_0 + \frac12 \sum_{k\in\Z^d\setminus\{0\}}c_k\cos(k\cdot x)
$$
for coefficients $c_k\in\R$ satisfying $c_k=c_{-k}$.
For $\mu\in\cP(\T^d)$ and $k\in \Z^d\setminus\{0\}$, set
\begin{equation}\label{eq:cos-sin-applications}
a_k(\mu)=\int_{\T^d} \cos(k\cdot x)\,\mu(\dx),\qquad b_k(\mu) := \int_{\T^d} \sin(k\cdot x)\,\mu(\dx).
\end{equation}
Then,
\begin{align*}
\iint_{\T^d\times\T^d} \cos(k\cdot (x-y))\,\mu(\dx)\mu(\dy)  &= \iint_{\T^d\times\T^d}[\cos(k\cdot x)\cos(k\cdot y) + \sin(k\cdot x)\sin(k\cdot y)]\,\mu(\dx)\mu(\dy)\\
&= a_k(\mu)^2+b_k(\mu)^2.
\end{align*}
Hence,
\begin{equation}\label{eq:F-kernel-Fourier}
\cF(\mu)= c_0 + \frac12 \sum_{k\neq0} c_k q_k(\mu),\qquad \text{where}\qquad q_k(\mu):=a_k(\mu)^2+b_k(\mu)^2,
\end{equation}
and by differentiating, a linear derivative is given by
\begin{equation}
\delta_\mu F(\mu,x)=   \sum_{k\neq0}c_k[a_k(\mu)\cos(k\cdot x)+b_k(\mu)\sin(k\cdot x)].
\end{equation}

We next characterize Lasry-Lions monotonicity using the Fourier expansion.

\begin{Prop}\label{prop:LL}
The coupling $f$ is Lasry-Lions monotone if and only if $c_k\ge 0$ for all $k\in\Z^d\setminus\{0\}$. In this case, $\U$ is the unique stationary equilibrium.
\end{Prop}

\begin{proof}
Using the Fourier expansion of $f(x,\mu_1)$ and $f(x,\mu_2)$, we readily obtain
$$
\int_{\T^d} (f(x,\mu_1)-f(x,\mu_2))\,(\mu_1-\mu_2)(\dx) =  \sum_{k\neq0} c_k\Big([a_k(\mu_1)-a_k(\mu_2)]^2 + [b_k(\mu_1)-b_k(\mu_2)]^2 \Big).
$$
If $c_k\geq0$ for all $k\neq0$, then this is clearly non-negative. Conversely, if $f$ is Lasry-Lions monotone, we claim that $c_k\geq0$ for every $k\in\Z^d\setminus\{0\}$. Indeed, fix $k\neq0$ and, for sufficiently small $\eps>0$, define two probability measures on $\T^d$ by 
$$
\mu_\pm(\dx) := \frac{1}{(2\pi)^d}\Big(1\pm \sqrt{\frac{\eps}{2}}\cos(k\cdot x)\Big)\,\dx.
$$
Then,
$$
a_\ell (\mu_+) - a_\ell(\mu_-) = \sqrt{\frac{\eps}{2}}\, \mathds{1}_{\ell\in\{-k,k\}},\qquad b_\ell(\mu_+)-b_\ell(\mu_-)=0,\qquad \ell\in \Z^d\setminus\{0\}.
$$
Therefore,
\begin{equation*}
0\leq \int_{\T^d} (f(x,\mu_+)-f(x,\mu_-))(\mu_+-\mu_-)(\dx) = \eps \, c_k.\qedhere
\end{equation*}
\end{proof}

\begin{Prop}
Let $(\eta_t)_{t\ge0}$ solve the heat equation $\partial_t\eta_t=\Delta\eta_t$ starting from some $\eta_0\in\cP^2_+(\T^d)$. Then,
\begin{equation}\label{eq:derivative-F-Fourier}
\frac{\d}{\dt} \cF(\eta_t) = -   \sum_{k\in\Z^d\setminus\{0\}}c_k |k|^2 q_k(\eta_t),\qquad t\ge0.
\end{equation}
Consequently, if
\begin{equation}\label{eq:Lambda}
\Lambda(\psi):=\sup\Big\{\frac{-1}{I(\mu)} \sum_{k\in\Z^d\setminus\{0\}} c_k|k|^2  q_k(\mu)\,\big|\,\mu\in\cP^2_+(\T^d)\setminus\{\U\}\Big\} < \frac{\kappa_c}{2},
\end{equation}
then the uniform distribution $\U$ is the unique stationary equilibrium.
\end{Prop}

\begin{proof}
We first prove \eqref{eq:derivative-F-Fourier}. Recall $a_k,b_k$ from \eqref{eq:cos-sin-applications}. We compute
\begin{align*}
\partial_t a_k(\eta_t) &=  \int_{\T^d} \cos(k\cdot x)\,\partial_t\eta_t(x)\,\dx\\
&=\int_{\T^d} \cos(k\cdot x)\,\Delta \eta_t(x)\,\dx\\
&= - |k|^2 \int_{\T^d} \cos(k\cdot x)\, \eta_t(x)\,\dx \qquad \text{(integration by parts)}\\
&= - |k|^2 a_k(\eta_t),
\end{align*}
and, similarly, $\partial_tb_k(\eta_t)=-|k|^2b_k(\eta_t)$.  Using the representation \eqref{eq:F-kernel-Fourier}, 
\begin{equation*}
\frac{\d}{\dt} \cF(\eta_t) = \frac12 \sum_{k\neq0} c_k \partial_t (a_k(\eta_t)^2+b_k(\eta_t)^2) =  -\sum_{k\neq0} c_k |k|^2 (a_k(\eta_t)^2+b_k(\eta_t)^2)= -\sum_{k\neq0} c_k |k|^2 q_k(\eta_t).
\end{equation*}
Combining this with Corollary \ref{cor:heat-flow},
$$
-\sum_{k\neq0} c_k |k|^2 q_k(\eta_0)= \frac{\d}{\dt}\Big|_{t=0} \cF(\eta_t) \geq \frac{\kappa_c}{2}I(\eta_0)
$$
whenever $\eta_0$ is a stationary equilibrium. This implies the claim.
\end{proof}

The proof of the next lemma is given in Appendix \ref{app:fisher-info-lower-bound}.

\begin{Lem}\label{lem:fisher-info-lower-bound}
For any $k\in\Z^d\setminus\{0\}$, 
$$
\sup_{\mu\in\cP^2_+(\T^d)\setminus\{\U\}} \, \frac{2|k|^2q_k(\mu)}{I(\mu)}=1.
$$
\end{Lem}

\begin{Cor}\label{cor:uniqueness-single-cos}
If $\psi(x)=\gamma \cos(k\cdot x)$ for some $k\in\Z^d\setminus\{0\}$ and  $\gamma\in\R$, then  $\Lambda(\psi)= \gamma_-$. Consequently, if $\gamma >-\frac{\kappa_c}{2}$, then $\U$ is the unique stationary equilibrium.
\end{Cor}

\begin{proof}
In this case $c_j=0$ for all $j\notin\{0,-k,k\}$ and $c_k=c_{-k}=\gamma $.  By definition of $\Lambda$ in \eqref{eq:Lambda},
$$
\Lambda(\psi)= \sup\Big\{\frac{-2\gamma |k|^2q_k(\mu)}{I(\mu)} \, \Big| \,\ \mu\in \cP^2_+(\T^d)\setminus\{\U\} \Big\}.
$$
If $\gamma \geq0$, then choose any $\mu\neq \U$ with $q_k(\mu)=0$ to see that $\Lambda(\psi)=0$ (see the proof of Proposition \ref{prop:LL} for a construction). For $\gamma <0$, by Lemma \ref{lem:fisher-info-lower-bound}, $\Lambda(\psi)=-\gamma=\gamma_-$.
\end{proof}

\section{Application to the Kuramoto MFG}\label{sec:Kuramoto}

The classical Kuramoto MFG of Carmona, Cormier, and Soner \cite{carmona_synchronization_2023}, first proposed by Yin, Mehta, Meyn, and Shanbhag in \cite{yin2011synchronization, yin2011bifurcation} and further studied in \cite{hofer2025synchronization, carmona2025kuramoto, cesaroni2024stationary}, provides a fascinating model to explain the phenomenon of spontaneous synchronization  as an outcome of a non-cooperative game between a continuum of agents. Here, the state space is the one-dimensional torus representing the phase of the typical \emph{oscillator} (or player) that minimizes \eqref{eq:typical-player-problem} with running cost given by
\begin{equation}\label{eq:Kuramoto-running-cost}
f_\kappa(x,\mu)  = 2\kappa \, \int_{\T} \sin^2\Big(\frac{x-y}{2}\Big)\,\mu(\dy),\qquad (x,\mu)\in\T\times\cP(\T).   
\end{equation}
Here, $\kappa>0$ is the \emph{coupling constant}. Higher values mean a higher cost of being misaligned. Clearly, $f_\kappa$ is the linear derivative of the potential
$$
\cF_\kappa(\mu) = \kappa\, \iint_{\T\times\T} \sin^2\Big(\frac{x-y}{2}\Big)\,\mu(\dx)\mu(\dy),\qquad \mu\in\cP(\T),
$$
see also H\"{o}fer and Soner \cite[Section 6.3]{hofer2024optimal}. 

Carmona, Cormier, and Soner establish a phase transition in the Kuramoto MFG. More precisely, they prove that in the \emph{supercritical regime} $\kappa>\kappa_c$, there exist non-uniform stationary equilibria, while the uniform distribution is a locally stable equilibrium for $\kappa<\kappa_c$, see \cite[Theorem 4.4]{carmona_synchronization_2023}. Here,  $\kappa_c$ is defined as in \eqref{eq:critical-kappa}. Using the results of our present paper, we are able to strengthen the analysis in the subcritical regime $\kappa<\kappa_c$. We also remark that in this case,
$$
\kappa_c=\sigma^2\Big(\rho+\frac{\sigma^2}{2}\Big),
$$
in accordance with the critical coupling strength in \cite{carmona_synchronization_2023} upon setting $\sigma:=\sqrt{2\nu}$. 

\begin{Cor}
Consider the Kuramoto MFG with running cost \eqref{eq:Kuramoto-running-cost} for $d=1$. Then, for all $\kappa<\kappa_c$, the uniform distribution $\U$ on $\T$ is the unique stationary Nash equilibrium. In addition, if $(\mu_t)_{t\ge0}$ is any time-dependent equilibrium, then $\mu_t$ converges to $\U$ weakly as $t\uparrow\infty$.
\end{Cor}

\begin{proof}
Observe that $\cF_\kappa$ is of the form \eqref{eq:F-kernel} with $\psi=\psi_\kappa$ where
$$
\psi_\kappa(z) = \kappa \sin^2\left(\frac{z}{2}\right)=\frac{\kappa}{2}-\frac{\kappa}{2}\cos(z),\qquad c_0=\frac{\kappa}{2},\qquad c_{-1}=c_1=-\frac{\kappa}{2}.
$$
By Corollary \ref{cor:uniqueness-single-cos}, $\Lambda(\psi_\kappa)=\frac{\kappa}{2}$, so that the claims follow from Corollary \ref{cor:uniqueness-single-cos} and Theorem \ref{thm:main}.
\end{proof}

\appendix

\section{HJB estimates}\label{app:HJB-estimates}

\begin{Lem}[$L^\infty$-bounds]\label{lem:L-infty-bounds}
Let $(u,m)$ be a classical solution to \eqref{eq:inhomog-MFG-system}. Then, there exists a finite constant $c_*>0$ such that
$$
\| \nabla u\|_{L^\infty([0,\infty)\times\T^d)}+\| \Delta u\|_{L^\infty([0,\infty)\times\T^d)}\leq c_*.
$$
\end{Lem}

\begin{proof}
Set $\psi(t,x):=f(x,\mu_t)$ for $(t,x)\in[0,\infty)\times\T^d$.
\vspace{1em}

\emph{Step 1. Gradient estimate.} Let $(t,x)\in[0,\infty)\times\T^d$ and $\xi\in\R^d$ be a direction (with $|\xi|=1$). For a control $\alpha\in\sA$, let 
$$
X^{t,x,\alpha}_s=x + \int_t^s \alpha_r\,\dr + \sqrt{2\nu}\, (W_s-W_t),\qquad s\ge t.
$$
Then,
$$
\frac{1}{|h|} |u(t,x+h\xi)-u(t,x)|\leq \sup_{\alpha\in\sA}\E\Big[\int_t^\infty e^{-\rho (s-t)}\frac{1}{|h|} |\psi(s,X^{t,x,\alpha}_s+h\xi)-\psi(s,X^{t,x,\alpha}_s)|\,\ds\Big].
$$
Since $\|\nabla \psi\|_\infty\leq \sup_{\mu\in\cP(\T^d)}\|\nabla f(\mu,\cdot)\|_\infty<\infty$, the claim follows from taking $h\downarrow0$ in
$$
\frac{1}{|h|} |u(t,x+h\xi)-u(t,x)| \leq \frac{1}{\rho} \|\nabla \psi\|_\infty.
$$

\emph{Step 2. Time derivative.} Since $m$ solves $\partial_t m=\nu\Delta m + \nabla\cdot(m\nabla u)$, we can differentiate
\begin{align*}
\partial_t\psi(t,x) &= \int_{\T^d} \delta_\mu f(x,\mu_t,y)\, \partial_tm(t,y)\,\dy\\
&= \int_{\T^d} \delta_\mu f(x,\mu_t,y)\, (\nu \Delta_y m+\nabla_y\cdot(m\nabla_y u))\,\dy\\
&= \int_{\T^d} (\nu \Delta_y \delta_\mu f(x,\mu_t,y) - \nabla_y \delta_\mu f(x,\mu_t,y) \cdot \nabla u(t,y))m\,\dy,
\end{align*}
so that 
$$
\sup_{(t,x)\in(0,\infty)\times \T^d}|\partial_t \psi(t,x)|\leq \nu\|\Delta_y \delta_\mu f\|_\infty + \rho^{-1}\|\nabla_y \delta_\mu f\|_\infty \|\nabla f\|_\infty =:c_0<\infty.
$$
For $(t,x)\in(0,\infty)\times\T^d$ and $h>0$ set
$$
w_h(t,x) := \frac{1}{h} (u(t+h,x)-u(t,x)).
$$
Then,
\begin{align*}
\frac1h (\psi(t+h,x)-\psi(t,x))&=
-\partial_t w_h + \rho w_h - \nu\Delta w_h +\frac{1}{2h}(|\nabla u(t+h,x)|^2-|\nabla u(t,x)|^2)\\
&= -\partial_t w_h + \rho w_h - \nu\Delta w_h + b_h(t,x)\cdot \nabla w_h(t,x),
\end{align*}
with $b_h(t,x):=(\nabla u(t+h,x)+\nabla u(t,x))/2$. By Step 1, $b_h$ is uniformly bounded. By Feynman-Kac,
$$
w_h(t,x)=\E\Big[\int_t^\infty e^{-\rho(s-t)} \frac{\psi(s+h,X^{t,x}_s)-\psi(s,X^{t,x}_s)}{h}\,\ds\Big],
$$
where $X^{t,x}_s=x-\int_t^s b_h(r,X^{t,x}_r)\,\dr +\sqrt{2\nu}\,(W_s-W_t)$.  Hence, by letting $h\downarrow0$,
$$
\|\partial_t u\|_{L^\infty([0,\infty)\times\T^d)}\leq \frac{1}{\rho} c_0.
$$

\emph{Step 3. Laplacian.} The estimate for the Laplacian follows from the HJB equation, together with the estimate $\|u\|_\infty\leq \|f\|_\infty/\rho$ from the control characterization. Indeed, the HJB equation yields
$$
\|\Delta u\|_{L^\infty([0,\infty)\times\T^d)} \leq \nu^{-1} \big(\|\partial_t u\|_\infty  + \rho \|u\|_\infty + \frac12 \|\nabla u\|_\infty^2 + \|f\|_\infty\big),
$$
which is uniformly bounded by the previous steps and Assumption \ref{asm:coupling}.
\end{proof}

\begin{Lem}[Stability]\label{lem:stability}
Let $\psi_n:[0,\infty)\times\T^d\mapsto\R$, $n\ge1$, and $\bar\psi:\T^d\mapsto\R$ be bounded continuously differentiable functions. Let $u_n:[0,\infty)\times\T^d\mapsto\R$ be the classical bounded solution to the discounted HJB
$$
-\partial_t u_n+\rho u_n - \nu\Delta u_n+\frac12|\nabla u_n|^2=\psi_n(t,x),\qquad (t,x)\in(0,\infty)\times\T^d,
$$
and let $\bar u:\T^d\mapsto\R$ be the unique classical solution to the stationary HJB
$$
\rho \bar u - \nu\Delta \bar u +\frac12|\nabla \bar u|^2=\bar \psi(x),\qquad x\in\T^d.
$$
If
\begin{itemize}
\item $\sup_{n\ge1} \big(\|\psi_n\|_{L^\infty([0,\infty)\times\T^d)} +  \|\partial_t \psi_n\|_{L^\infty([0,\infty)\times\T^d)} + \|\nabla \psi_n\|_{L^\infty([0,\infty)\times\T^d)} \big)<\infty$,
\item  $\psi_n\to\bar\psi$ locally uniformly as $n\uparrow\infty$,
\end{itemize}
then $\nabla u_n\to\nabla \bar u$ locally uniformly as $n\uparrow\infty$.
\end{Lem}

\begin{proof}
Set
$$
C:=\max\Big\{\sup_{n\geq1}\|\psi_n\|_{L^\infty([0,\infty)\times\T^d)},\ \|\bar\psi\|_{L^\infty(\T^d)}\Big\}.
$$
For $r\ge0$, we let $Q_r:=[0,r]\times \T^d$.
\vspace{1em}

\emph{Step 1. Convergence of $u_n$.} Using the control characterization of $u_n$ and $\bar u$, for any $T>0$ and $(t,x)\in[0,\infty)\times\T^d$,
\begin{align*}
|u_n(t,x)-\bar u(x)| &\leq \sup_\alpha \E\int_t^\infty e^{-\rho (s-t)} |\psi_n(s,X^{\alpha}_s)-\bar \psi(X^\alpha_s)|\,\d s \\
&\leq \sup_{s\in[t,t+T]}\, \|\psi_n(s,\cdot)-\bar \psi\|_\infty \int_t^{t+T} e^{-\rho(u-t)}\,\d u  + \int_{t+T}^\infty e^{-\rho(s-t)}2C\,\d s\\
&\leq \frac1\rho \sup_{s\in[t,t+T]}\, \|\psi_n(s,\cdot)-\bar \psi\|_\infty + \frac{2C}{\rho} e^{-\rho T}
\end{align*}
This implies, for any $t\ge0$,
$$
\|u_n-\bar u\|_{L^\infty(Q_{t})} \leq \frac{1}{\rho}\|\psi_n-\bar\psi\|_{L^\infty(Q_{t+T})}+\frac{2C}{\rho} e^{-\rho T}.
$$
Sending first $n\uparrow\infty$ and then $T\uparrow \infty$ shows that $u_n\to\bar u$ locally uniformly.
\vspace{1em}

\emph{Step 2.} Set 
$$
w_n := u_n-\bar u,\qquad f_n:=\psi_n-\bar\psi,\qquad b_n:=\frac12 (\nabla u_n+\nabla \bar u).
$$
Using the dynamic programming equations satisfied by $u_n$ and $\bar u$, 
$$
-\partial_t w_n + \rho w_n - \nu \Delta w_n + b_n\cdot\nabla w_n=f_n,\qquad \text{on}\ \ (0,\infty)\times\T^d.
$$
The same argument as in the proof of Lemma \ref{lem:L-infty-bounds} shows that
$$
M:=\sup_n \|b_n\|_{L^\infty([0,\infty)\times\T^d)}<\infty.
$$
Fix $T>0$, and choose $\eta \in C^\infty_c([0,T+1])$ with $0\le \eta\le 1$ and $\eta\equiv1$ on $[0,T]$. Define 
$$
v_n(t,x):=\eta(t)w_n(t,x),\qquad (t,x)\in Q_{T+1}.
$$
Then $v_n=w_n$ on $Q_T$, $v_n(T+1,\cdot)\equiv0$, and 
$$
-\partial_t v_n + \rho v_n - \nu \Delta v_n + b_n\cdot \nabla v_n = g_n,\qquad \text{on} \ \ Q_{T+1}.
$$
where $g_n:=\eta f_n-\eta ' w_n$. Note that 
$$
\|g_n\|_{L^\infty(Q_{T+1})}\leq \|\psi_n-\bar\psi\|_{L^\infty(Q_{T+1})} + \|\eta'\|_{L^\infty([0,T+1))} \, \|u_n-\bar u\|_{L^\infty(Q_{T+1})},
$$
showing that $\|g_n\|_{L^\infty(Q_{T+1})}\to0$ as $n\uparrow\infty$ by assumption and Step 1.

\vspace{1em}
\emph{Step 3.} Let us reverse time and set, for $(s,x)\in Q_{T+1}$,
$$
(\tilde v_n,\tilde b_n,\tilde g_n)(s,x) := (v_n,b_n,g_n)(T+1-s,x),\qquad h_n:=\tilde g_n-\tilde b_n\cdot\nabla \tilde v_n-\rho \tilde v_n.
$$
Then, $\tilde v_n$ satisfies
$$
\partial_s \tilde v_n-\nu\Delta  \tilde v_n=h_n,\qquad \tilde v_n(0,\cdot)=0.
$$
Hence, by Feynman-Kac, for $(s,x)\in Q_{T+1}$
\begin{align*}
\tilde v_n(s,x) &= \E\left[\int_0^s h_n(r,x+\sqrt{2\nu}W_{s-r})\,\dr\right] \\
&= \int_0^s \int_{\T^d} h_n(r,x+y)K(s-r,y)\,\d y\,\dr\\
&=\int_0^s \int_{\T^d} h_n(r,y)K(s-r,x-y)\,\d y\,\dr.
\end{align*}
Here, $K(\cdot,\cdot)$ is the periodic heat kernel
$$
K(t,y):=\sum_{k\in (2\pi\Z)^d} G(t,y+k),\qquad (t,y)\in(0,\infty)\times\T^d
$$
where
$$
G(t,y) = \frac{1}{(4\pi\nu t)^{d/2}}\ \exp\left(-\frac{|y|^2}{4\nu t}\right),\qquad (t,y)\in(0,\infty)\times\R^d.
$$
Differentiating, 
$$
\|\nabla K(t,\cdot)\|_{L^1(\T^d)} \leq \sum_{k\in (2\pi\Z)^d} \int_{\T^d} |\nabla G(t,y+k)|\,\dy = \int_{\R^d} |\nabla G(t,y)|\,\d y= C_{\nu,d} t^{-1/2}
$$
for some $C_{\nu,d}<\infty$ only depending on $\nu$ and the dimension $d$.  We obtain
$$
\|\nabla \tilde v_n(s,\cdot)\|_{L^\infty(\T^d)} \leq C_{\nu,d}\int_0^s\frac{1}{\sqrt{s-r}} \left(\|\tilde g_n\|_{L^\infty(Q_{T+1})} + M\|\nabla \tilde v_n(r,\cdot)\|_{L^\infty(\T^d)} + \rho \|\tilde v_n\|_{L^\infty(Q_{T+1})}\right)\,\dr.
$$
Letting $x_n(s):=\|\nabla \tilde v_n(s,\cdot)\|_{L^\infty(\T^d)} $,
\begin{equation}\label{eq:volterra}
x_n(s)\leq a_n + b \int_0^s \frac{1}{\sqrt{s-r}} x_n(r)\,\dr ,\qquad s\in[0,T+1],   
\end{equation}
with 
$$
a_n:= 2 C_{\nu,d} \sqrt{T+1} \, (\|\tilde g_n\|_{L^\infty(Q_{T+1})} + \rho \|\tilde v_n\|_{L^\infty(Q_{T+1})}),\qquad b:= C_{\nu,d}M.
$$
\emph{Step 4.} We iterate the estimate \eqref{eq:volterra}, for $s\in[0,T+1]$,
\begin{align*}
x_n(s) &\leq a_n+ b\int_0^s \frac{1}{\sqrt{s-r}}\left(a_n + b \int_0^r \frac{1}{\sqrt{r-\tau}} x_n(\tau)\,\d\tau\right)\,\d r\\
&= a_n + 2a_n b \sqrt{s} + b^2 \int_0^s \underbrace{\left( \int_\tau^ s \frac{1}{\sqrt{s-r}} \frac{1}{\sqrt{r-\tau}}\,\dr\right)}_{=\pi} x_n(\tau)\,\d \tau\\
&\leq  a_n (1+ 2  b \sqrt{T+1}) + b^2 \pi \int_0^s  x_n(\tau)\,\d \tau.
\end{align*}
By Gr\"{o}nwall's inequality, $\|x_n\|_{L^\infty([0,T+1])}\leq C_T a_n$ with $C_T:=(1+2b\sqrt{T+1})\exp(b^2 \pi (T+1))$. Hence,
$$
\|\nabla v_n\|_{L^\infty(Q_{T+1})} = \|\nabla \tilde v_n\|_{L^\infty(Q_{T+1})} = \|x_n\|_{L^\infty([0,T+1])} \leq C_T' (\|g_n\|_{L^\infty(Q_{T+1})} + \rho \|v_n\|_{L^\infty(Q_{T+1})}).
$$
where $C_T':=2C_TC_{\nu,d} \sqrt{T+1}$.
Now $\nabla v_n=\nabla w_n$ on $Q_T$, so
\begin{align*}
\|\nabla u_n-\nabla \bar u\|_{L^\infty(Q_{T})} &= \|\nabla w_n\|_{L^\infty(Q_{T})}\\
&=\|\nabla v_n\|_{L^\infty(Q_{T})}\\
&\leq \|\nabla v_n\|_{L^\infty(Q_{T+1})}\\
&\leq C_T' (\|g_n\|_{L^\infty(Q_{T+1})} + \rho \|v_n\|_{L^\infty(Q_{T+1})}).
\end{align*}
The right-hand side converges to zero as $n\uparrow\infty$ by Step 2, showing the claim.
\end{proof}

\section{Proof of Lemma \ref{lem:torus-inequality}} \label{app:torus-inquality}

Define $v:=\sqrt{g}>0$. Then
$$
\int_{\T} |(\log g)''|^2g\,\d x
=4\int_{\T}\Big(v''-\frac{(v')^2}{v}\Big)^2\dx.
$$
By periodicity of $v$ and $v'$,
$$
0=\int_{\T}\Big(\frac{(v')^3}{v}\Big)'\dx
=3\int_{\T}\frac{(v')^2v''}{v}\,\dx
-\int_{\T}\frac{(v')^4}{v^2}\,\dx,\qquad \Rightarrow\qquad \int_{\T}\frac{(v')^2v''}{v}\,\d x
=\frac13\int_{\T}\frac{(v')^4}{v^2}\,\d x.
$$
Expanding the square,
$$
\int_{\T} |(\log g)''|^2g\,\dx
=4\int_{\T}|v''|^2\,\d x
+\frac43\int_{\T}\frac{|v'|^4}{v^2}\,\dx
\ge 4\int_{\T}|v''|^2\,\dx.
$$
By periodicity, $\int_\T v'\,\dx=0$, so that by the Poincare inequality on $\T$,
$$
\int_{\T}|v'|^2\,\d x\leq \int_{\T}|v''|^2\,\dx.
$$
Therefore,
$$
\int_{\T} |(\log g)''|^2g\,\dx
\ge 4\int_{\T}|v'|^2\,\d x
=\int_{\T} |(\log g)'|^2g\,\dx,
$$
as claimed.

\section{Proof of Lemma \ref{lem:fisher-info-lower-bound}}
\label{app:fisher-info-lower-bound}

\begin{proof}
Let $X\sim \mu$, set $Y:=(k\cdot X)\,\text{mod} \, 2\pi$ and $\eta=\cL(Y)$. Denote the densities by
$$
m(x):=\frac{\d\mu}{\dx}(x),\qquad \varrho(y):=\frac{\d\eta}{\dy}(y),\qquad x\in\T^d,\\ y\in\T.
$$

\vspace{1em}

\emph{Step 1. $I(\mu)\geq|k|^2 I(\eta)$}.
For any smooth test function $\varphi:\T\mapsto\R$, we compute, using integration by parts and the equality $\varrho'=(\log(\varrho))'\varrho$,
\begin{align*}
\int_\T \varphi'\varrho\,\dy = -\int_\T\varphi\varrho'\,\dy = - \int_\T\varphi (\log(\varrho))'\varrho\,\dy.
\end{align*}
At the same time,
\begin{align*}
\int_\T \varphi'\varrho\,\dy &= \int_{\T^d} \varphi'(k\cdot x)m(x)\,\dx\\
&= \frac{1}{|k|^2}\int_{\T^d} k\cdot \nabla_x (\varphi (k\cdot x)) m(x)\,\dx\\
&= - \frac{1}{|k|^2}\int_{\T^d} \varphi (k\cdot x)\, k\cdot \nabla_x m(x)\,\dx \qquad \text{(integration by parts)}\\
&= - \frac{1}{|k|^2}\int_{\T^d}  \varphi (k\cdot x)\, k\cdot \nabla_x \log (m(x))\, m(x)\,\dx \\
&= - \frac{1}{|k|^2}\E[ \varphi (k\cdot X)\, k\cdot \nabla_x \log (m(X))] \\
&= - \frac{1}{|k|^2}\E[ \varphi (Y) \, \E[ k\cdot \nabla_x \log (m(X))\,|\,Y]] \qquad \text{(law of iterated expectations)}\\
&=- \int_\T \varphi(y) \left(\frac{1}{|k|^2}\, \E[ k\cdot \nabla_x \log (m(X))\,|\,Y=y]\right) \varrho(y)\,\dy.
\end{align*}
Since $\varphi$ was arbitrary, $\log(\varrho)'\varrho=|k|^{-2}\, \E[ k\cdot \nabla_x \log (m(X))\,|\,Y=\cdot ]\varrho$. Then,
\begin{align*}
I(\eta)&= \int_{\T} |(\log(\varrho))'|^2\varrho\,\dy \\
&\leq \frac{1}{|k|^4}\int_{\T} \E[|k\cdot \nabla_x \log(m(X))|^2\,|\,Y=y]\,\varrho(y)\,\dy\qquad \text{(Jensen's inequality)}\\
&=\frac{1}{|k|^4} \int_{\T^d} |k\cdot \nabla_x\log(m(x))|^2\,m(x)\,\dx\\
&\leq \frac{1}{|k|^2 }\int_{\T^d} |\nabla_x\log(m(x))|^2\,m(x)\,\dx \\
&=\frac{1}{|k|^2} I(\mu).
\end{align*}

\vspace{1em}

\emph{Step 2.} We show that $I(\eta)\geq 2|\int e^{iy}\varrho\,\dy|^2$.  Define $(r,\phi)$ by
$$
re^{i\phi} = \int_{\T} e^{iy}\varrho(y)\,\dy,
$$
and note that 
$$
r=\mathrm{Re}\left(\int_\T e^{i(y-\phi)}\varrho(y)\,\dy\right)= \int_\T\cos(y-\phi)\varrho(y)\,\dy,\qquad \text{and}\qquad r^2=\Big|\int_\T e^{iy}\varrho\,\dy\Big|^2.
$$
Introduce a new density $\zeta$ on $\T$ by
$$
\zeta(y) := \frac{e^{2r\cos(y-\phi)}}{Z},\qquad Z=\int_{\T} e^{2r\cos(y-\phi)}\,\dy.
$$
Using the inequality $x\log x\geq x-1$ for $x\ge0$,
$$
\int_\T\log(\varrho)\varrho\,\dy - \int_\T\log(\zeta)\varrho\,\dy= \int_\T \log\left(\frac{\varrho}{\zeta}\right)\varrho\,\dy  \geq \int_\T \left(\frac{\varrho}{\zeta}-1\right)\zeta\,\dy=0. 
$$
Hence,
$$
\int_\T \log(\varrho)\varrho\,\dy\geq\int_\T \log(\zeta)\varrho\,\dy=\int_\T(2r\cos(y-\phi))\varrho(y)\,\dy -  \log Z=2r^2- \log Z\geq r^2-\log(2\pi).
$$
The last inequality used
$$
\frac{1}{2\pi}Z=\frac{1}{2\pi}\sum_{n\in 2\Z_+}  \frac{(2r)^n}{n!}\int_\T\cos(y)^n\,\dy=  \sum_{n\in 2\Z_+}  \frac{(2r)^n}{n!} \frac{n!}{2^n((n/2)!)^2} =  \sum_{n\ge0} \frac{r^{2n}}{(n!)^2}\leq \sum_{n\ge0}\frac{r^{2n}}{n!}=e^{r^2}.
$$
The logarithmic Sobolev inequality on $\T$, see for example {\'E}mery and Yukich \cite{emery1987simple}, reads
$$
\U(f^2\log(f^2)) \leq \U(f^2)  \log(\U(f^2)) + 2 \U(|f'|^2),
$$
where $\U$ is the uniform distribution on $\T$ for a sufficiently smooth function $f:\T\mapsto\R$. Applying this to $f=\sqrt{2\pi\varrho}$ yields
$$
\int_\T \log(2\pi\varrho) \varrho\,\dy \leq \frac12 \int_\T \frac{|\varrho'|^2}{\varrho}\,\dy = \frac12 I(\eta).
$$
Hence,
$$
I(\eta)\geq 2 \int_\T \log(2\pi \varrho)\varrho\,\dy \geq 2r^2.
$$

\vspace{1em}
\emph{Step 3.} Combining the previous steps  we obtain 
$$
I(\mu)\geq |k|^2 I(\eta)\geq 2|k|^2 \ \Big|\int_\T e^{iy}\varrho\,\dy\Big|^2 = 2|k|^2 q_k(\mu).
$$
This establishes
$$
\sup_{\mu\neq\U}\ \frac{2|k|^2 q_k(\mu)}{I(\mu)}\leq 1.
$$

\vspace{1em}
\emph{Step 4.} 
It remains to show that this is sharp. For $\epsilon\in(0,1/4)$, define
$$
m_\eps(x):=\frac{1}{(2\pi)^d}(1+2\eps\cos(k\cdot x)) = \frac{1}{(2\pi)^d}(1+\eps(e^{ik\cdot x}+e^{-ik\cdot x})),\qquad x\in\T^d.
$$
Let $\mu_\eps(\dx)=m_\eps(x)\,\dx$. We see that $m_\eps$ is a positive probability density and compute
$$
q_k(\mu_\eps) = \Big|\int_{\T^d} e^{ik\cdot x}m_\eps(x)\,\d x\Big|^2 = \Big|\frac{1}{(2\pi)^d}\int_{\T^d}[e^{ik\cdot x}+\eps (1+e^{2ik\cdot x})]\,\dx\Big|^2=\eps^2,
$$
using $\int e^{ik\cdot x}\,\dx=0$ for $k\neq 0$. Next, 
$$
\nabla m_\eps(x) = - \frac{2\eps}{(2\pi)^d} \sin(k\cdot x)\, k,
$$
so that
$$
I(\mu_\eps) = \int_{\T^d} \frac{|\nabla m_\eps|^2}{m_\eps}\,\dx = \frac{4\eps^2|k|^2}{(2\pi)^d}\int_{\T^d} \frac{\sin^2(k\cdot x)}{1+2\eps\cos(k\cdot x)}\,\dx.
$$
By bounded convergence,
$$
\frac{I(\mu_\eps)}{2|k|^2\eps^2} = \frac{2}{(2\pi)^d}\int_{\T^d} \frac{\sin^2(k\cdot x)}{1+2\eps\cos(k\cdot x)}\,\dx \ \xrightarrow[\ \ \eps\downarrow0 \ \ ]{}\ \frac{2}{(2\pi)^d}\int_{\T^d} \sin^2(k\cdot x)\,\dx=1,
$$
as claimed.
\end{proof}

{\small
\bibliographystyle{abbrvnat}
\bibliography{references}
}

\end{document}